\documentclass{amsart}

\usepackage{amsmath,amssymb,amsthm}

\newtheorem{prop}{Proposition}[section]
\newtheorem{thm}[prop]{Theorem}
\newtheorem{lem}[prop]{Lemma}

\theoremstyle{definition}

\newtheorem{rem}[prop]{Remark}
\newtheorem*{nota}{Notation}

\newcommand{\bfone}{\mathbf{1}}

\newcommand{\bfa}{\mathbf{a}}
\newcommand{\alphast}{\alpha_{\mathrm{st}}}

\newcommand{\rmd}{\mathrm{d}}

\newcommand{\bfe}{\mathbf{e}}

\newcommand{\lk}{\mathtt{lk}}
\newcommand{\lambdac}{\lambda_{\mathrm{c}}}
\newcommand{\lambdas}{\lambda_{\mathrm{s}}}

\newcommand{\bfn}{\mathbf{n}}

\newcommand{\R}{\mathbb{R}}

\newcommand{\SO}{\mathrm{SO}}
\newcommand{\ttsl}{\mathtt{sl}}

\newcommand{\bfv}{\mathbf{v}}

\newcommand{\xist}{\xi_{\mathrm{st}}}

\newcommand{\Z}{\mathbb{Z}}

\DeclareMathOperator{\Int}{Int}

\begin{document}

\author[S.~Durst]{Sebastian Durst}
\author[H.~Geiges]{Hansj\"org Geiges}
\address{Mathematisches Institut, Universit\"at zu K\"oln,
Weyertal 86--90, 50931 K\"oln, Germany}
\email{sdurst@math.uni-koeln.de}
\email{geiges@math.uni-koeln.de}

\author[J. Gonzalo]{Jes\'us Gonzalo P\'erez}
\address{Departamento de Matem\'aticas, Universidad Aut\'onoma
de Madrid, 28049 Madrid, Spain}
\email{jesus.gonzalo@uam.es}

\author[M.~Kegel]{Marc Kegel}
\address{Institut f\"ur Mathematik, Humboldt-Universit\"at zu Berlin,
Unter den Linden 6, 10099 Berlin, Germany}
\email{kegemarc@math.hu-berlin.de}

\thanks{H.~G.\ is partially supported by the SFB/TRR 191
``Symplectic Structures in Geometry, Algebra and Dynamics''.
M.~K.\ is supported by the Berlin Mathematical School.}

\title[Parallelisability of $3$-manifolds]{Parallelisability of $3$-manifolds
via surgery}

\date{}

\begin{abstract}
We present two proofs that all closed, orientable $3$-manifolds
are parallelisable. Both are based on the Lickorish--Wallace surgery
presentation; one proof uses a refinement of this presentation
due to Kaplan and some basic contact geometry. This complements
a recent paper by Benedetti--Lisca.
\end{abstract}



\maketitle


\section{Introduction}
Our aim in this note is to present two proofs of the following
well-known theorem.

\begin{thm}
\label{thm:main}
Every closed, orientable $3$-manifold is parallelisable.
\end{thm}

In a recent paper~\cite{beli18}, Benedetti and Lisca give three
new ``bare hands'' proofs of this result. They also sketch
four earlier proofs. We refer to their paper for the history of
these proofs and a discussion of their relative merits.

Our first proof of Theorem~\ref{thm:main} is in some sense
the most elementary yet. Its basic ingredient is the Lickorish--Wallace
surgery presentation of $3$-manifolds. By modifying a given
framing along surfaces (with boundary) transverse to one vector
in the $3$-frame, one can ensure the extendability
of an initial framing on the $3$-sphere over the surgery tori.
Some simple linear algebra allows us to show that the required modifications
are possible.

This proof is implicit in an earlier paper~\cite{gego97} of two of the present
authors. The contact circles on $3$-manifolds constructed there
in particular give rise to a parallelisation. Apparently, this aspect
has gone largely unnoticed, so it seems opportune to extract
a completely self-contained proof of parallelisability from that
paper.

Our second proof of Theorem~\ref{thm:main} is less elementary,
relying as it does on a refined surgery presentation due to
Kaplan and some contact topology. However, this proof compensates
for it by being extremely short. Kaplan's result can be proved
by some clever but simple moves in Kirby diagrams. The contact topology
used is essentially the fact that transverse knots in contact $3$-manifolds
have odd self-linking number. We only need this for knots transverse
to the standard contact structure in the $3$-sphere, and we indicate a direct
proof of the result in this particular case.
\section{A surgical proof}
\label{section:surgery}
\subsection{The surgery presentation}
Let $M$ be a closed, orientable $3$-manifold. According to
the surgery presentation theorem of Lickorish \cite{lick62}
and Wallace~\cite{wall60}, cf.\ \cite[Section~9.I]{rolf76},
$M$ can be constructed as follows.

Start with a solid torus $S^1\times D^2$, embedded in
the $3$-sphere $S^3$ in the standard way.
Denote a meridian $*\times
\partial D^2$ by $\lambda_{\infty}$ and a longitude $S^1\times *$
(with $*\in\partial D^2$) by $\sigma_{\infty}=\mu_{\infty}$. In using
this
notation we think of $\lambda_{\infty},\mu_{\infty}$ as longitude and
meridian of a complementary solid torus $S^3\setminus\Int(S^1\times D^2)$.
In particular, the longitude $\mu_{\infty}$ is homologically
trivial in this complementary solid torus.

Choose a pure braid of unknots $K_i$, $i=1,\ldots, n$, in $S^1\times D^2$,
that is, each $K_i$ is the graph of a map $S^1\rightarrow\Int(D^2)$.
Let $N_i$ be disjoint tubular neighbourhoods
of the~$K_i$. The solid tori $N_i$ may be assumed to intersect
each slice $\{*\}\times D^2$ in a single disc. Write
\[ M_0:= (S^1\times D^2)\setminus\bigcup_{i=1}^n \Int(N_i)\]
for the $3$-manifold with boundary obtained by cutting out
the interiors of these solid tori.

On the boundary $2$-torus $\partial N_i$ there is a (homologically)
well-defined meridian $\mu_i$ that generates the kernel of
$\pi_1(\partial N_i)\rightarrow\pi_1(N_i)$.
Moreover, there is a distinguished longitude $\lambda_i$ on $\partial N_i$,
characterised by being homologically trivial in $S^3\setminus\Int(N_i)$.
The $\lambda_i$ and $\mu_i$ generate~$\pi_1(\partial N_i)$.

Now choose simple closed curves $\sigma_i$ on $\partial N_i$ homologous to
$\mu_i\pm \lambda_i$, $i=1,...,n$. Glue solid tori
$V_i\cong S^1\times D^2$, $i=1,...,n,\infty$, to the
boundary components of $M_0$ such that $\sigma_i$
becomes a meridian in $V_i$. For the appropriate choices in this
construction, the resulting $3$-manifold will be diffeomorphic
to the given~$M$.
\subsection{Extending framings over solid tori}
Write $\theta$ for the coordinate in the $S^1$-factor
of the solid torus $S^1\times D^2$, and $x,y$ for cartesian
coordinates on the $D^2$-factor. On $S^1\times D^2$ we have the
parallelisation given by
\[ \bfe_1=\partial_{\theta},\;\;\bfe_2=\partial_x,\;\;\bfe_3=\partial_y.\]

Observe that when we perturb this frame slightly to make
$\bfe_1$ transverse to the boundary torus $T^2=\partial(S^1\times D^2)$,
the pair $(\bfe_2,\bfe_3)$ induces a framing of $T^2$ by
projection along~$\bfe_1$.
Relative to the Lie group framing of $T^2$, this induced frame
makes no twist along the longitude $\mu_{\infty}$, and
a single twist along the meridian~$\lambda_{\infty}$. The
\emph{Lie group framing} on $T^2=S^1\times S^1$ is the one obtained
by transporting a basis at the identity element under the
left-action of the Lie group. An explicit example is $(\partial_{\varphi},
\partial_{\theta})$, where $\varphi$ is the angular coordinate in the
$xy$-plane.

The fact that $\pi_1(\SO_3)=\Z_2$ and $\pi_2(\SO_3)=\{1\}$
translates into the following extendability condition.

\begin{lem}
\label{lem:extend}
Let $(\bfv_1,\bfv_2,\bfv_3)$ be a trivialisation of
$T(S^1\times D^2)$ along the boundary $T^2=\partial(S^1\times D^2)$,
with $\bfv_1$ transverse to the boundary. This trivialisation extends over
$S^1\times D^2$ if and only if the framing of $T^2$ induced by
$(\bfv_2,\bfv_3)$ makes an odd number of turns with respect to
the Lie group framing of $T^2$ along some (and hence any)
meridian $\{*\}\times\partial D^2$.
\qed
\end{lem}

Now perturb the parallelisation of $M_0$ given by
$(\bfe_1,\bfe_2,\bfe_3)$ so as to make $\bfe_1$ transverse to
the boundary $\partial M_0$. Then, with respect to
the Lie group framing of the boundary tori, the induced framing
makes no turn along the longitudes
$\lambda_1,\ldots,\lambda_n,\mu_{\infty}$, and $\pm 1$ turns
along the meridians $\mu_1,\ldots,\mu_n$. Hence, the framing
makes $\pm 1$ turns along the surgery curves $\sigma_1,\ldots,\sigma_n$,
and no turn along the surgery curve $\sigma_{\infty}$.

\begin{nota}
From now on, $(\bfe_1,\bfe_2,\bfe_3)$ will denote this
slightly perturbed frame with $\bfe_1$ transverse to~$\partial M_0$.
\end{nota}

We therefore need to change the parity of the twisting along
$\sigma_{\infty}$ without affecting the parity along the
other surgery curves. We shall achieve this by introducing
additional twists into the framing along surfaces (with boundary)
transverse to~$\bfe_1$.
\subsection{Surfaces in $M_0$}
A surface $\Sigma\subset M_0$ with boundary is called
\emph{properly embedded} if it is embedded, its boundary $\partial\Sigma$
lies on~$\partial M_0$, and $\Sigma$ is transverse to~$\partial M_0$.

We begin with the following simple observation.

\begin{lem}
\label{lem:twist}
Let $\Sigma\subset M_0$ be a properly embedded surface.
Let $(\bfv_1,\bfv_2,\bfv_3)$ be a parallelisation of
$M_0$ with $\bfv_1$ transverse to~$\Sigma$.
Then there is a framing $(\bfv_1,\bfv'_2,\bfv'_3)$
of $M_0$ that coincides with $(\bfv_1,\bfv_2,\bfv_3)$ outside a
tubular neighbourhood of $\Sigma$ and such that $(\bfv'_2,\bfv'_3)$
makes one additional twist about $\bfv_1$ as one passes~$\Sigma$.
\qed
\end{lem}

Let $\Sigma_{\infty}$ be the element in $H_2(M_0,\partial M_0)$
(with integral coefficients understood)
represented by $(*\times D^2)\cap M_0$.
We shall allow ourselves to identify $\Sigma_{\infty}$ and
other elements of $H_2(M_0,\partial M_0)$ with a particular
surface representing this class.

By the definition of $\lambda_i$ there is an element $\Sigma_i$
in $H_2(M_0,\partial M_0)$ represented by
an annulus with boundary curves $\mu_{\infty}$ and $\lambda_i$,
and with a certain number of discs removed where the $N_j$,
$j\neq i$, cut this annulus.

Denote the linking number of the unknots $K_i$
and $K_j$ (the souls of $N_i$ and $N_j$,
respectively) by~$l_{ij}$; here $K_i$ and $K_j$ are regarded
as knots in~$S^3$. Only the value of $l_{ij}$ modulo~2
is relevant to our problem, so we need not worry about a
sign convention for this linking number.

For $\Sigma$ a surface properly embedded in $M_0$,
and $\sigma$ a curve in $\partial M_0$ transverse to~$\partial\Sigma$,
we write $\# (\sigma ,\Sigma )$ for
their intersection number, in other words,
the intersection number of $\sigma$ and $\partial\Sigma$
as curves on~$\partial M_0$. Again we only count this modulo~2.

Now consider an element $\Sigma\in H_2(M_0,\partial M_0)$ of the form
\[ \Sigma =\sum_{i=1}^n a_i\Sigma_i+a_{\infty}\Sigma_{\infty}. \]

\begin{lem}
\label{lem:Sigma}
For $|a_{\infty}|$ sufficiently large, $\Sigma$ can be represented
by a union of $n$ properly embedded but not necessarily
disjoint surfaces transverse to~$\bfe_1$.
\end{lem}

\begin{proof}
We may consider the $\Sigma_i$ individually, and for our application
it suffices to consider the cases $a_i=0$ and $a_i=1$, and
$a_{\infty}>0$; the general case is analogous.
For $a_i=0$ we can take $a_{\infty}=1$. For
$a_i=1$, the union of $\Sigma_i$ and $a_{\infty}$ disjoint
parallel copies of $\Sigma_{\infty}$ is homologically equivalent
to a helicoidal surface in $T^2\times [0,1]$
(with discs removed where it intersects the other~$N_j$),
where $T^2\times\{0\}\equiv
\partial N_i$ and $T^2\times\{1\}\equiv S^1\times\partial D^2\subset
\partial M_0$. By increasing $a_{\infty}$, the slope of this
helicoidal surface can be made as small as necessary to achieve
transversality with~$\bfe_1$.
\end{proof}

\subsection{Modifying the framing of~$M_0$}
Modulo $2$ we have the following intersection numbers.
\[ \begin{array}{lcl}
\# (\sigma_{\infty},\Sigma_{\infty}) & = & 1,\\
\# (\sigma_{\infty},\Sigma_i) & = & 0,\;\;\;i=1,\ldots,n,
\end{array} \]
and, for $i=1,\ldots,n$,
\[ \begin{array}{lcl}
\# (\sigma_i,\Sigma_{\infty}) & = & 1,\\
\# (\sigma_i,\Sigma_i) & = & 1,\\
\# (\sigma_i,\Sigma_j) & = & l_{ij}\;\;\; \mbox{\rm for}\;\; i\neq j.
\end{array} \]

Now take a collection $\Sigma$ of surfaces as in Lemma~\ref{lem:Sigma},
and introduce a twist into the frame $(\bfe_1,\bfe_2,\bfe_3)$
as in Lemma~\ref{lem:twist} along each component of~$\Sigma$.
Since we want to change the parity of the number of twists
along~$\sigma_{\infty}$, we require (modulo~$2$)
\[ a_{\infty}=1. \]
Along the $\sigma_i$ we need to keep the parity unchanged,
which translates into
\[ a_i+\sum_{j\neq i}l_{ij}a_j+a_{\infty}=0.\]
Theorem~\ref{thm:main} is then an immediate consequence of the
following lemma.

\begin{lem}
The linear system of equations
\[ a_i+\sum_{j\neq i}l_{ij}a_j=1\]
over $\Z_2$, where $l_{ij}=l_{ji}$, always has a solution.
\end{lem}

\begin{proof}
Set $l_{ii}=1$, and let $L$ be the
$(n\times n)$-matrix $(l_{ij})$. Write $\bfa$ for the column vector
with entries $a_1,...,a_n$. Let $\bfone$ be the column vector with
all $n$ entries equal to~$1$. Then the linear system of equations can
be written as $L\bfa=\bfone$. To show
that this system has a solution for any choice of $L$ (with
$l_{ij}=l_{ji}$ and $l_{ii}=1$) it is sufficient to show that any
linear relation satisfied by the rows of $L$ is also satisfied by the
rows (i.e.~entries) of~$\bfone$. This reduces to showing that
the sum of $k$ rows of $L$ can only be the zero row if $k$ is even.
By exchanging the columns it suffices to show this for the sum of the
first $k$ rows. If this sum is the zero row, then in
particular $\sum_{i,j=1}^kl_{ij}=0$, hence
$k\equiv\sum_{i<j}2l_{ij}=0$.
\end{proof}
\section{A contact geometric proof}
\subsection{Contact geometry background}
The standard contact structure on $S^3\subset\R^4$ is the
tangent $2$-plane field $\xist$ defined as the kernel of the $1$-form
\[ \alphast=x_1\,\rmd y_1-y_1\,\rmd x_1+x_2\,\rmd y_2-y_2\,\rmd x_2\]
(restricted to the tangent bundle $TS^3$).
A framing of $S^3$ adapted to this contact structure is given by
\begin{eqnarray*}
\bfe_1 & = & x_1\partial_{y_1}-y_1\partial_{x_1}+x_2\partial_{y_2}-
             y_2\partial_{x_2},\\
\bfe_2 & = & x_1\partial_{x_2}-x_2\partial_{x_1}+y_2\partial_{y_1}-
             y_1\partial_{y_2},\\
\bfe_3 & = & x_1\partial_{y_2}-y_2\partial_{x_1}+y_1\partial_{x_2}-
             x_2\partial_{y_1}.
\end{eqnarray*}
Here $\bfe_1$ is transverse to~$\xist$; the vector fields $\bfe_2,\bfe_3$
span the plane field~$\xist$.

The plane field $\xist$ is maximally non-integrable, that is,
the defining $1$-form $\alphast$ has the property that
$\alphast\wedge\rmd\alphast$ is a volume form on~$S^3$.
A consequence of this non-integrability is that any knot in
$S^3$ can be $C^0$-approximated by a knot transverse
to $\xist$ and isotopic to the original
knot~\cite[Theorem~3.3.1]{geig08}.

The \emph{self-linking number} $\ttsl(K)$ of a knot $K\subset S^3$ transverse
to $\xist$ is defined as follows. Push $K$ in the direction of
the vector field $\bfe_2$ trivialising $\xist$ to obtain a parallel
copy $K'$ of~$K$. Then set $\ttsl (K)$ equal to the linking number $\lk(K,K')$
of $K$ and $K'$. (A chosen orientation of $K$ defines one of~$K'$;
the value of $\lk(K,K')$ is independent of this choice.)
\subsection{Surgery along a single knot}
Suppose we want to perform surgery along a knot $K$ with
integral surgery coefficient~$n$. By the preceding section we may
assume that $K$ is transverse to~$\xist$. Then $(\bfe_2,\bfe_3)$
defines a framing of~$K$, i.e.\ a trivialisation of its normal
bundle, corresponding to the parallel knot~$K'$. We think of $K'$
as a longitude $\lambdac$ on the boundary of a closed tubular neighbourhood
$\nu K$ of~$K$. We also consider a longitude $\lambdas\subset\partial(\nu K)$
corresponding to the surface framing of~$K$, that is, the framing defined
by an embedded compact, orientable surface in $S^3$ with boundary~$K$.
Writing $\mu$ for the meridian of $\nu K$, we have
\[ \lambdac=\lambdas+\ttsl(K)\cdot\mu\]
by the definition of $\ttsl(K)$.

Surgery along $K$ with coefficient $n\in\Z$ then means that we cut out
$\Int(\nu K)$ and glue in a solid torus $S^1\times D^2$ by sending
its meridian $*\times\partial D^2$ to
\[ n\mu+\lambdas=(n-\ttsl(K))\cdot\mu+\lambdac.\]

\begin{lem}
The framing $(\bfe_1,\bfe_2,\bfe_3)$ on $S^3\setminus\Int(\nu K)$
extends over the glued-in torus if and only if $n-\ttsl(K)$ is odd.
\end{lem}

\begin{proof}
Analogous to the considerations in Section~\ref{section:surgery}
we may assume that, after a small perturbation, $\bfe_1$ is transverse
to~$\partial(\nu K)$, and $(\bfe_2,\bfe_3)$ induces a framing
of $\partial(\nu K)$ that makes a single twist, relative to
the Lie group framing, along the meridian $\mu$
and no twist along the longitude $\lambdac$. Now apply
Lemma~\ref{lem:extend}.
\end{proof}
\subsection{Proof of Theorem~\ref{thm:main}}
We now use the following fact from contact topology,
see~\cite[Remark~4.6.35]{geig08}.

\begin{lem}
The self-linking number of any knot $K\subset S^3$ transverse to
$\xist$ is odd.
\qed
\end{lem}

\begin{rem}
The proof of this result for transverse knots in arbitrary contact
$3$-manifolds is quite similar to that of the Poincar\'e--Hopf index theorem,
but it requires a good understanding of surfaces in contact $3$-manifolds.
For transverse knots in $(S^3,\xist)$ it follows more easily by
relating transverse knots to Legendrian knots (i.e.\ knots tangent to~$\xist)$;
specifically, combine Propositions 3.5.23 and 3.5.36 from~\cite{geig08}.
A quite direct proof, only using Reidemeister moves for knot
projections, can be given analogous to the proof of
\cite[Prop.~3.5.23]{geig08}.
\end{rem}

Thus, a given closed, orientable $3$-manifold will inherit
a parallelisation, provided it can be represented by a surgery diagram
in $S^3$ with all surgery coefficients \emph{even}. The existence of such a
surgery diagram has been established by
Kaplan~\cite{kapl79}, cf.~\cite[Theorem~9.4]{foma97}, \cite[pp.~191--2]{gost99}
or~\cite{kege18}.

\begin{rem}
A purely topological proof of Theorem~\ref{thm:main} based on Kaplan's
theorem goes as follows. First one shows that the $4$-dimensional
handlebody $W$ obtained by gluing $2$-handles to a $4$-ball $D^4$ along
a link in $S^3=\partial D^4$ with even coefficients is
parallelisable~\cite[Theorem~9.3]{foma97}. From this one finds
a quasi-framing of $M=\partial W$, i.e.\ a framing of
$M\setminus\Int(D^3)$, and then a framing of $M$ by observing that
the relevant obstructions lie in trivial groups, cf.~\cite{beli18}.
As an alternative for this second step, one can use a $4$-frame on $W$
to define a quaternionic structure $(I,J,K)$ on the tangent bundle $TW$;
the outward normal $\bfn$ along $M=\partial W$ then yields
a parallelisation of $M$: simply project $I\bfn,J\bfn,K\bfn$
onto $TM$ along~$\bfn$.

A clever replacement for Kaplan's theorem is proved
in~\cite{beli18}, from which one can then conclude as above.
\end{rem}
%
%
%
%
%
%

\end{document}